\documentclass[12pt,oneside,final]{amsart}
\title{Introduction to inverse problems for hyperbolic PDEs}

\IfFileExists{tweakslo.sty}{\usepackage{tweakslo}}{
\usepackage{amssymb,thmtools,mathtools,todonotes}
\declaretheorem{theorem}\declaretheorem{lemma}\declaretheorem{corollary}}
\newcommand{\HOX}[1]{\todo[noline,color=white,size=\footnotesize]{#1}}

\def\p{\partial}
\def\R{\mathbb R}
\DeclareMathOperator{\supp}{supp}
\DeclarePairedDelimiter\norm{\lVert}{\rVert}

\usepackage{hyperref}
\usepackage{graphicx}
\usepackage{comment}
\usepackage{subfig}

\author[]{Medet Nursultanov}
\address {Department of Mathematics and Statistics, University of Helsinki, Helsinki, and Institute of Mathematics and Mathematical Modeling, Almaty, Kazakhstan}
\email{medet.nursultanov@gmail.com}

\author[]{Lauri Oksanen}
\address {Department of Mathematics and Statistics, University of Helsinki, Helsinki, Finland}
\email{lauri.oksanen@helsinki.fi}

\begin{document}\maketitle

\tableofcontents

These lecture notes were written for CIRM SMF School {\em Spectral Theory, Control and Inverse Problems}, November 2022.

\section{Introduction}

There are two main approaches to solve inverse {\em coefficient determination} problems for wave equations: the Boundary Control method and an approach based on geometric optics. 
These notes focus on the Boundary Control method, but we will have a brief look at the geometric optics as well.

\subsection{Boundary Control method}

The Boundary Control method can be used to solve several coefficient determination problems for wave equations.
The method originates from \cite{Belishev87}, and it allows for building a rather comprehensive theory of inverse problems of coefficient determination type. Indeed, many such problems
can be reduced to inverse boundary spectral problems, 
which again can be reduced to coefficient determination problems for wave equations. This is the case for inverse problems for heat and non-stationary Schro\"odinger equations \cite{Mandache04},
as well as for several time-fractional \cite{Kian18} and space-fractional \cite{Feizmohammadi} equations. Inverse problems for linear elliptic equations on a wave guide \cite{DSFerreira16}, and some (very special) non-linear elliptic equations \cite{Carstea} are also covered by this theory. 

The method has been generalized for symmetric systems of wave equations with scalar leading order \cite{Kurylev18} and for wave equations on Lorentzian manifolds satisfying a curvature bound \cite{Alexakis22}.

We mention that the Boundary Control method has been extended to certain non-smooth settings. For instance, see \cite{KianMoranceyOksanen,LiuSaksalaYan} for non-smooth potentials and \cite{AndersonKatsudaKurylev,IsozakiKurylevLassas2015,IsozakiKurylevLassas2017,KirpichnikovaKurylev} for non-smooth geometries.

In these notes we give an introductory exposition of the method. For an in-depth presentation, we refer to the monograph \cite{Matti}. 

\subsection{Geometric optics based approach}

Geometric optics constructions are well-suited for solving inverse problems with time-dependent coefficients. In the case of time-independent coefficients, this approach has not reached the same geometric generality as the Boundary Control method.  

The geometric optics based approach goes back to \cite{StefanovZ}. It has been generalized for wave equations on real analytic Lorentzian manifolds satisfying a convex foliation condition \cite{SY, Stefanov1}, as well as for certain static and stationary  Lorentzian manifolds \cite{FIKO,FIO}.  We also mention work \cite{SylvesterUhlmann} with the precursor \cite{Calderon}, where complex geometric optics used to solve analogous inverse problems for elliptic partial differential equations. 

At the moment it is not clear which of the two methods works better in the Lorentzian context. This context is natural since it respects the idea that coordinates do not exist a priori in nature, and hence should play no role in the mathematical formulation of physical models.

\section{Boundary Control method in $1+1$ dimensions}

Consider the following  initial-boundary value problem
\begin{equation}\label{main eq lec 2}
	\begin{cases}
		(\partial_t^2 - \partial_x^2 + q) u = 0; & \text{for } (t,x) \in \mathbb{R}^+\times (0,1);\\
		\left. u\right|_{x=0} = f;\\
		\left. u\right|_{x=1} = 0;\\
		\left. u\right|_{t=0} = \left. \partial_t u\right|_{t=0} = 0,
	\end{cases}
\end{equation}
where $f$ is a boundary source, $q$ is a potential, and $\R^+$ stands for the interval $(0, \infty)$. 
For simplicity, we will assume that $q \in C^\infty(0,1)$, a smooth function up to the end points of the interval $[0,1]$, and that $f \in C_0^\infty(\R^+)$, a smooth compactly supported function. Then \eqref{main eq lec 2} has a unique solution that satisfies $u \in C^\infty(\R^+ \times (0,1))$, see e.g. the theory in Section 7.2 of \cite{Evans1998}.

We write $u^f=u$ for the solution to \eqref{main eq lec 2} to emphasize its dependence on the source $f$. The Dirichlet-to-Neumann map is a natural choice of data for inverse problems as it directly relates boundary inputs to boundary outputs. Therefore, we consider the inverse problem with the Dirichlet-to-Neumann map as data, defined as follows:
\begin{equation}\label{DtN}
	\Lambda f = \left. \partial_x u^f \right|_{x=0}.
\end{equation}
We will study the following inverse problem: determine the potential $q$ given the operator $\Lambda$.

\subsection{Finite speed of propagation in the case $q = 0$}

The finite speed of propagation property for the acoustic wave equation plays a central role in both direct and inverse theory. 
To illustrate the property in the simplest possible case, let us 
consider \eqref{main eq lec 2} with $q = 0$.
We set
\begin{equation*}
	\mathcal{E}(t,x)  = |\partial_t u(t,x)|^2 + |\partial_x u(t,x)|^2.
\end{equation*}

\begin{theorem}[Conservation of energy]
	Suppose that $u \in C^2(\overline{{\R}^+\times (0,1)})$ satisfies
	\begin{equation*}
		\begin{cases}
			(\partial_t^2 - \partial_x^2) u = 0 \quad \text{on } \mathbb{R}^+\times (0,1)\\
			\left.u\right|_{x=0} = 
\left.u\right|_{x=1} = 0.
		\end{cases}
	\end{equation*}
	Then the global energy
	\begin{align*}
		E(t) = \frac{1}{2} \int_{0}^{1} \mathcal{E}(t,x) dx
	\end{align*}
	satisfies $\partial_t E(t) = 0$ for all $t > 0$.
\end{theorem}
\begin{proof}
	We write,
	\begin{multline*}
		\partial_t E(t) =  \frac{1}{2} \int_{0}^{1} \partial_t \mathcal{E} (t,x)dx \\
		=  \int_{0}^{1}\left(  \partial_t u(t,x) \partial_t^2 u(t,x) + \partial_x u(t,x) \partial_{t}\partial_{x} u(t,x)\right) dx.
	\end{multline*}
	By integration by parts, we obtain
	\begin{equation*}
		\partial_t E(t) =  \left[\partial_x u(t,x)\partial_t u(t,x)\right]^{1}_{0}
		+ \int_{0}^{1} \left( \partial_t^2u(t,x) - \partial_x^2u(t,x)\right) \partial_t u(t,x) dx.
	\end{equation*}
	The first term is zero since $u(t,0) = u(t,1) = 0$, and the second term is zero since $u$ satisfies the wave equation.
\end{proof}

\begin{theorem}[Finite speed of propagation]
	Suppose that $u \in C^2(\overline{{\R}^+\times (0,1)}) $ satisfies
	\begin{equation}\label{eq_wave_simple}
		\begin{cases}
			(\partial_t^2 - \partial_x^2) u = 0 \quad \text{on } \mathbb{R}^+\times (0,1)\\
			\left.u\right|_{x=1} = 0.
		\end{cases}
	\end{equation}
	Then the local energy
	\begin{align*}
		E(t) = \frac{1}{2} \int_t^1 \mathcal{E}(t,x) dx 
	\end{align*}
	satisfies $\partial_t E(t) \le 0$ for all $t > 0$.
\end{theorem}
\begin{proof}
	We write, by using Leibniz integral rule,
	\begin{align*}
		\partial_t E(t) &= - \frac{1}{2}  \mathcal{E}(t, t) + \frac{1}{2} \int_{t}^{1} \partial_t \mathcal{E} (t,x)dx \\
		& = - \frac{1}{2}  \mathcal{E}(t, t) + \int_{t}^{1}\left(  \partial_t u(t,x) \partial_t^2 u(t,x) + \partial_x u(t,x) \partial_{t}\partial_{x} u(t,x)\right) dx.
	\end{align*}
	An integration by parts gives
	\begin{multline*}
		\partial_t E(t) = - \frac{1}{2}  \mathcal{E}(t, t) + \left[\partial_x u(t,x)\partial_t u(t,x)\right]^{1}_{t}\\
		+ \int_{t}^{1} \left( \partial_t^2u(t,x) - \partial_x^2u(t,x)\right) \partial_t u(t,x) dx.
	\end{multline*}
	Since $u$ is the solution of the wave equation the last integral is $0$. Moreover, since $u(t,1) = 0$, the second term is $0$ at $1$, so that
	\begin{equation*}
		\partial_t E(t) = - \frac{1}{2}  \mathcal{E}(t, t) - \partial_tu(t,t)\partial_x u(t,t) = - \frac{1}{2}\left(\partial_tu(t,t) + \partial_x u(t,t)\right)^2\leq 0.
	\end{equation*}
\end{proof}

In particular, if $E(0) = 0$ then $E(t) = 0$ for all $t > 0$.
So if $u$ and $\partial_t u$ vanish initially on $(0,1)$ then at time $t$ they vanish for $x \in (t,1)$.
To see that this statement is optimal, observe that 
	\begin{align*}
u(t, x) = f(t - x) - f(t + x - 2)
	\end{align*}
is a solution to \eqref{eq_wave_simple} for any $f \in C^2(\R)$. Moreover, if $f(r)$ vanishes for $r < 0$, then $E(0) = 0$. 
But $u$ does not vanish near the line $t = x$ if $f(r)$ does not vanish for $r > 0$ near $r=0$.

\subsection{Finite speed of propagation for general $q$}

Let us now consider a wave equation with a potential as in \eqref{main eq lec 2}, and prove a finite speed of propagation result in the very natural setting of a diamond, see the set $K$ in Theorem~\ref{th_fsp_1d} below. 
The set $K$ is visualized in Figure \ref{1d}.

It is also possible to study the case where boundary conditions are posed, say at $x=0$ and $x=1$ as before, and where the diamond intersects the boundary of $\R^+ \times (0,1)$. However, we leave this case to the reader.

\begin{theorem}[Finite speed of propagation]\label{th_fsp_1d}
	Let $X>0$ and define
	\begin{equation*}
		K = \{(t,x)\in \mathbb{R}\times\mathbb{R}: X- |x|\geq |t|\}.
	\end{equation*}
Let $q \in L^\infty(K)$ and let $u \in C^2(K)$ be a solution of the equation
	\begin{equation*}
		\begin{cases}
			\left(\partial_t^2 - \partial_x^2 + q(t,x)\right) u = 0, & \text{on } K;\\
			\left.u\right|_{t=0} = \left.\partial_t u\right|_{t=0} = 0, & \text{on } (-X,X).
		\end{cases}
	\end{equation*}
Then $u = 0$ in $K$.
\end{theorem}
\begin{proof}
We will consider the case $t>0$. The case $t < 0$ is analogous and we omit its proof. 
		Let us set $I(t) = [-X + t, X - t]$ and define the energy
		\begin{equation*}
			E(t)= \frac{1}{2} \int_{I(t)} \mathcal{E}(t,x) dx.
		\end{equation*}
		Using Leibniz integral rule, we obtain
		\begin{align*}
			\partial_t E(t) =& -\frac{1}{2}\left(  |\partial_t u(t,  X - t)|^2 + |\partial_x u(t, X - t)|^2 \right)\\
			& -\frac{1}{2}\left(  |\partial_t u(t, - (X - t))|^2 + |\partial_x u(t,  - (X - t))|^2 \right)\\
			& + \int_{I(t)} \left( \partial_t u(t,x) \partial_t^2u(t,x) + \partial_x u(t,x) \partial_{t} \partial_x u(t,x) \right) dx.
		\end{align*}
		An integration by parts gives
		\begin{align*}
		\partial_t E(t) =& -\frac{1}{2}\left(  |\partial_t u(t,  X - t)|^2 + |\partial_x u(t, X - t)|^2 \right)\\
		& -\frac{1}{2}\left(  |\partial_t u(t, - (X - t))|^2 + |\partial_x u(t,  - (X - t))|^2 \right)\\
		& + \partial_t u(t,X - t)\partial_x u(t, X - t) - \partial_t u(t, - (X - t))\partial_x u(t, - (X - t))\\
		& + \int_{I(t)} \partial_t u(t,x) \left(  \partial_t^2u(t,x) - \partial_x^2 u(t,x) \right) dx.
		\end{align*}
		Therefore,
		\begin{align*}
			\partial_t E(t) =& -\frac{1}{2}\bigg(  |\partial_t u(t, X - t)| - |\partial_x u(t, X - t)| \bigg)^2\\
			& -\frac{1}{2}\bigg(  |\partial_t u(t,  - (X - t))| + |\partial_x u(t, - (X - t))| \bigg)^2\\
			& + \int_{I(t)} \partial_t u(t,x) \left(  \partial_t^2u(t,x) - \partial_x^2 u(t,x) \right) dx.
		\end{align*}
		Since the first two terms are non-positive, it follows
		\begin{multline*}
			\partial_t E(t) \leq - \int_{I(t)} \partial_t u(t,x) q(t,x) u(t,x) dx\\ + \int_{I(t)} \partial_t u(t,x) \left(  \partial_t^2u(t,x) - \partial_x^2 u(t,x) +q(t,x)u(t,x) \right) dx.
		\end{multline*}
		Let us use the notation
		\begin{equation*}
			P= \partial_{t}^2 - \partial_x^2 + q.
		\end{equation*}
		The Cauchy-Schwarz inequality imply
		\begin{multline*}
			E(t) = E(0) + \int_{0}^{t} \partial_t E(s) ds\\
			\leq E(0) + 2\int_0^t \int_{I(s)}\left( |\partial_t u(s,x)|^2 +|q(s,x) u(s,x)|^2 + |P u(s,x)|^2 \right)dx ds.
		\end{multline*}
		Hence,
		\begin{equation}\label{basic_energy}
			E(t) \leq E(0) + C \left( \int_0^t E(s)ds + \int_0^t \int_{I(s)} \left(|u(s,x)|^2 + |Pu(s,x)|^2\right)dxds\right),
		\end{equation}
		where $C$ is a constant depending on $q$ and $X$, which will change from line to line. 
		
		Let us now take care of the $|u|^2$ term caused by the potential $q$. We will use arguments close to the proof in \cite[Section 4.2]{Ladyzhenskaya1985}.
		%
		%
		For $x\in \mathbb{R}$ fixed, we know
		\begin{equation*}
			u(t,x) = u(0,x) + \int_{0}^{t} \partial_{t}u(s,x) ds.
		\end{equation*}
		Let us square it, and then, use Cauchy's inequality and the Cauchy-Schwarz inequality to obtain
		\begin{align}\label{ineq_timederiv}
			|u(t,x)|^2 \le 2 |u(0,x)|^2 + 2 t \int_0^t |\partial_t u(s,x)|^2 ds.
		\end{align}
		Next, we define
		\begin{equation*}
			z(t) = E(t) +  \int_{I(t)} |u(t,x)|^2 dx.
		\end{equation*}
		Using \eqref{basic_energy} and \eqref{ineq_timederiv}, we estimate
		\begin{multline*}
			z(t)  \leq  E(0) + C \left( \int_0^t E(s)ds + \int_0^t \int_{I(s)} \left( | u(s,x)|^2 + |P u(s,x)|^2 \right)dxds\right)\\
			+ 2 \int_{I(t)} |u(0,x)|^2 dx + 2 t\int_{I(t)} \int_0^t |\partial_t u(s,x)|^2 ds dx.
		\end{multline*}
		Since we consider $t\in [0, X]$, and since $I(t) \subset I(s)$ for $s \in [0,t]$, the last term can be estimated from above by $2X \int_{0}^{t} E(s) ds$. Therefore, we have
		\begin{equation*}
			z(t)  \leq  Cz(0) + C \int_{0}^{t} z(s) ds + C\int_0^t \int_{I(s)} |P u(s,x)|^2dxds.
		\end{equation*}
Using the Gronwall's inequality in the integral form, see e.g. \cite[Appendix B.2.k]{Evans1998}, we obtain
	\begin{align}\label{finite_speed_1d}
z(t) \le C \left(z(0) + \int_0^t \int_{I(s)} |P u(s,x)|^2dxds\right).
	\end{align}
		Since $z(0) = 0$ and $Pu = 0$, we conclude that $z(t) = 0$ for $t\in [0,X]$. Recalling the definition of the function $z$, we see that $u(t,x) = 0$ for $t\in [0,X]$ and $x\in I(t)$.
\end{proof}

\begin{theorem}[Unique continuation]\label{uniq_con_1d}
	Let $T>0$ and define 
	\begin{equation*}
		K= \{(t,x)\in \mathbb{R}\times\mathbb{R}: |x|\leq T - |t|\}.
	\end{equation*}
Let $q \in L^\infty(K)$ and let $u \in C^2(K)$ be a solution of the equation
	\begin{equation*}
		\begin{cases}
			\left(\partial_t^2 - \partial_x^2 + q(t,x)\right) u = 0, & \text{on } K;\\
			\left.u\right|_{x=0} = \left.\partial_x u\right|_{x=0} = 0, & \text{on } (-T,T).
		\end{cases}
	\end{equation*}
	Then $u = 0$ on $K$.
\end{theorem}

\begin{figure}%
	\centering
	\subfloat[\centering Finite speed of the wave propagation]{{\includegraphics[width=6.2cm]{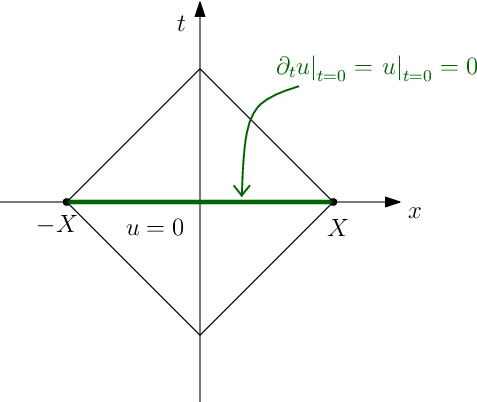} }}%
	\qquad
	\subfloat[\centering Unique continuation]{{\includegraphics[width=6.2cm]{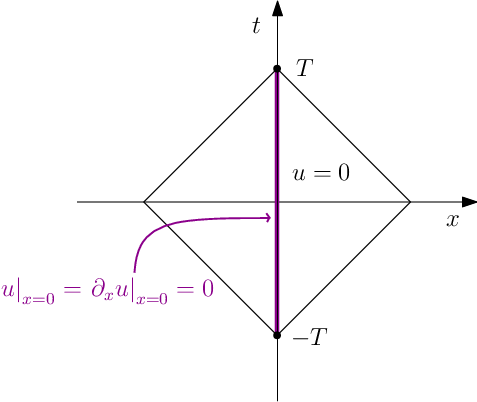} }}%
	\caption{$1 + 1$ dimensional case}%
	\label{1d}%
\end{figure}


\begin{proof}
By interchanging the roles of $t$ and $x$, we see that the theorem coincides with Theorem \ref{th_fsp_1d}; see Figure \ref{1d}.
\end{proof}

\subsection{Approximate controllability}
Let us recall that $u^f$ is the solution of \eqref{main eq lec 2}, when the boundary source is given by $f$. 
The following result is obtained by transposing unique continuation.

\begin{lemma}[Approximate controllability]\label{app con}
	Let $0< s< T\leq 1$, then the set 
	\begin{equation*}
		\mathcal{B}(s,T)= \left\{u^f(T,\cdot): f\in C_0^\infty(T-s,T)\right\}
	\end{equation*}
	is a dense subset of $L^2(0,s)$.
\end{lemma}

Here we use the identification
	\begin{align*}
L^2(0,s) = \{\phi \in L^2(0,1) : \supp(\phi) = (0,s)\}.
	\end{align*}
Note that due to finite speed of propagation the function $u^f(T, \cdot)$ is supported on $[0,s]$. Hence we may view $\mathcal{B}(s,T)$ as a subspace of $L^2(0,s)$.

\begin{proof}
	To show the density, it is enough to prove that $\mathcal{B}(s,T)^{\perp} = \{0\}$, where
	\begin{equation*}
		\mathcal{B}(s,T)^{\perp} = \left\{v\in L^2(0,s): (v,u)_{L^2(0,s)} = 0 \text{ for all } u \in 	\mathcal{B}(s,T)\right\}.
	\end{equation*}
	Let $h\in \mathcal{B}(s,T)^{\perp}$. Let $\omega$ be the solution of the equation
	\begin{equation*}
		\begin{cases}
			\left( \partial_t^2 - \partial_x^2 + q(x)\right)\omega(t,x) = 0 & \text{on } (0,T)\times (0,1);\\
			\left.\omega\right|_{x=0,1} = 0;\\
			\left.\omega\right|_{t=T} = 0;\\
			\left.\partial_t \omega\right|_{t=T} = h.
		\end{cases}
	\end{equation*}
	Let $f\in C_0^\infty(T-s,T)$. An integration by parts, together with the boundary and initial conditions for $u^f$ and $\omega$, give
	\begin{multline*}
		0 = \left(\left( \partial_t^2 - \partial_x^2 + q\right)u^f, \omega\right)_{L^2((0,T)\times(0,1))} - \left(u^f, \left( \partial_t^2 - \partial_x^2 + q\right)\omega\right)_{L^2((0,T)\times(0,1))}\\
		= -\int_0^1u^f(T,x)\partial_t\omega(T,x) dx - \int_0^Tu^f(t,0) \partial_x\omega(t,0) dt.
	\end{multline*}
	Since $\left.\partial_t \omega\right|_{t=T} = h\in \mathcal{B}(s,T)^{\perp}$, the first term of the right-hand side is zero, so that
	\begin{equation*}
	0 = \int_0^T f(t) \partial_x\omega(t,0) dt,
	\end{equation*}
	which is true for arbitrary $f\in C_0^\infty(T-s,T)$. Hence, $\partial_x\omega(t,0) = 0$ on $(T - s, T)$. 
	
	Let $\tilde{\omega}$ be the odd extension of $\omega$ to $(0, 2T)\times (0,1)$, more precisely,
	\begin{equation*}
		\tilde{\omega}(t,x)=
		\begin{cases}
			\omega(t,x) & \text{if } t\leq T\\
			-\omega (2T - t, x)  & \text{otherwise}.
		\end{cases}
	\end{equation*}
	Then $\tilde{\omega}$ satisfies
	\begin{equation*}
		\begin{cases}
			\left( \partial_t^2 - \partial_x^2 + q(x)\right)\tilde{\omega}(t,x) = 0 & \text{on } (0,2T)\times (0,1);\\
			\left.\tilde{\omega}\right|_{x=0} = 0 & \text{on } (0,2T);\\
			\left.\partial_x\tilde{\omega}\right|_{x=0} = 0 & \text{on } (T - s,T + s).
		\end{cases}
	\end{equation*}
	By unique continuation, see Theorem \ref{uniq_con_1d}, we obtain that $\tilde{\omega} = 0$ on 
	\begin{equation*}
		\{(t,x) \in (T - s, T + s)\times (0,1): |x| \leq s - |T - t|\}.
	\end{equation*}
	In particular, $\partial_t\omega(T,x) = h(x) = 0$ on $(0,s)$, so that $\mathcal{B}(s,T)^{\perp} = \{0\}$.
\end{proof}

\subsection{An integration by parts trick for the inverse problem}

Let us define the function
\begin{equation*}
	W_{f,h}(t,s) = (u^f(t,\cdot), u^h(s,\cdot) )_{L^2(0,1)}.
\end{equation*} 
Then the following holds
\begin{lemma}\label{scalar lec2}
	Let $f$, $h\in C_0^\infty(\R^+)$. The operator $\Lambda$ determines $W_{f,h}(t,t)$ for all $t>0$.
\end{lemma}
\begin{proof}
	Since $u^f$ and $u^h$ are solutions of \eqref{main eq lec 2}, there holds
	\begin{multline*}
		(\partial_t^2 - \partial_s^2) W_{f,h}(t,s) = \left( \partial_x^2 u^{f}(t,\cdot) - q u^f(t,\cdot), u^{h}(s,\cdot) \right)_{L^2(0,1)}\\
		- \left( u^{f}(t,\cdot), \partial_x^2 u^{h}(s,\cdot) - q u^h(s,\cdot) \right)_{L^2(0,1)}.
	\end{multline*}
The two terms with $q$ cancel out. 
	Further, integration by parts gives 
	\begin{multline*}
		(\partial_t^2 - \partial_s^2) W_{f,h}(t,s)  = \partial_xu^f(t,1) u^h(s,1) -  \partial_xu^f(t,0) u^h(s,0)\\
		-  u^f(t,1) \partial_xu^h(s,1) +  u^f(t,0) \partial_x u^h(s,0).
	\end{multline*}
	Boundary conditions for $u^f$ and $u^h$ give
	\begin{equation}\label{sec_der_of_prod}
		(\partial_t^2 - \partial_s^2)W_{f,h}(t,s) = f(t) \Lambda h(s) - \Lambda f(t) h(s).
	\end{equation}
	Let us denote the right-hand side of \eqref{sec_der_of_prod} by $F(t,s)$, then $W_{f,h}$ solves
	\begin{equation*}
		\begin{cases}
			(\partial_t^2 - \partial_s^2) W_{f,h}(t,s) = F(t,s) & \text{on } \R^+ \times \R^+,\\
			W_{f,h}(0,s) = \partial_tW_{f,h}(0,s) = 0.
		\end{cases}
	\end{equation*}
	Hence, $W_{f,h}$ is determined by $F$, and consequently, it is determined by $\Lambda$. 
(When evaluating $W_{f,h}(t,t)$, we do not need to impose any boundary condition at $s=0$ due to finite speed of propagation. Nonetheless, we could impose, for example, $W_{f,h}(t,0) = 0$ as this holds in view of the initial conditions satisfied by $u^h(s,x)$.)
\end{proof}

Consider two potentials $q_1$ and $q_2$. Let us write $u_1^f$ and $u_2^f$ for the solutions of \eqref{main eq lec 2} with $q$ replaced by $q_1$ and $q_2$, respectively. The corresponding Dirichlet-to-Neumann operators, see \eqref{DtN}, are denoted by $\Lambda_1$ and $\Lambda_2$. 
It follows from the above lemma that 
	\begin{align*}
(u_1^f(t,\cdot), u_1^h(t,\cdot) )_{L^2(0,1)}
= 
(u_2^f(t,\cdot), u_2^h(t,\cdot) )_{L^2(0,1)}
	\end{align*}
whenever $\Lambda_1 = \Lambda_2$.

\subsection{Solution to the inverse problem}

We will show that $\Lambda_1=\Lambda_2$ implies $q_1=q_2$. 
For a set $S \subset \R$, we define the indicator function
	\begin{align*}
1_S(x) = \begin{cases}
1 & x \in S, \\
0 & \text{otherwise}.
\end{cases}
	\end{align*}

\begin{lemma}\label{lem_of_norm_conv_of_sol_1d}
Assume that $\Lambda_1 = \Lambda_2$.
	Let $0<s<T$, let $f \in C_0^\infty(\R^+)$ and let
	\begin{align*}
f_j \in C_0^\infty(T-s, T), \quad j=1,2,\dots,
	\end{align*}
be a sequence such that
	\begin{equation}\label{conv_cutoff}
		u_1^{f_j}(T,\cdot) \rightarrow 1_{(0,s)}u_1^f(T,\cdot) \qquad \text{in } L^2(0,1).
	\end{equation}
	Then
	\begin{equation*}
		u_2^{f_j}(T,\cdot) \rightarrow 1_{(0,s)}u_2^f(T,\cdot) \qquad \text{in } L^2(0,1).
	\end{equation*}
\end{lemma}

\begin{proof}
Let $\tilde f \in C_0^\infty(T-s, T)$ and let $k=1,2$. We compute
\begin{align}\label{k1d0}
	\nonumber&\left\|u_k^{\tilde f}(T,\cdot) - u_k^{f}(T,\cdot)\right\|_{L^2(0,1)}^2 = \left\|u_k^{\tilde f}(T,\cdot) - 1_{(0,s)}u_k^{f}(T,\cdot)\right\|_{L^2(0,1)}^2 \\
	&\qquad + \left\|\left(1_{(0,s)} - 1\right) u_k^{f}(T,\cdot)\right\|_{L^2(0,1)}^2\\
	\nonumber&\qquad + 2\left(u_k^{\tilde f}(T,\cdot) - 1_{(0,s)}u_k^{f}(T,\cdot), \left(1_{(0,s)} - 1\right) u_k^{f}(T,\cdot)\right)_{L^2(0,1)}.
\end{align}
Due to finite speed of propagation, we know that $u_k^{\tilde f}(T,\cdot)$ is supported in $(0,s)$. Therefore, the functions
\begin{equation*}
	u_k^{\tilde f}(T,\cdot) - 1_{(0,s)}u_k^{f}(T,\cdot) 
	\qquad 
	 \left(1_{(0,s)} - 1\right) u_k^{f}(T,\cdot)
\end{equation*}
have disjoint supports, so that \eqref{k1d0} becomes
\begin{multline}\label{k1d}
	\left\|u_k^{\tilde f}(T,\cdot) - u_k^{f}(T,\cdot)\right\|_{L^2(0,1)}^2 = \left\|u_k^{\tilde f}(T,\cdot) - 1_{(0,s)}u_k^{f}(T,\cdot)\right\|_{L^2(0,1)}^2 \\
	+ \left\|\left(1_{(0,s)} - 1\right) u_k^{f}(T,\cdot)\right\|_{L^2(0,1)}^2.
\end{multline}
By approximate controllability, 
	\begin{align}\label{inf_proj}
\inf_{\tilde{f}\in C_0^\infty(T - s,T)} \left\| u_k^{\tilde{f}}(T,\cdot) - u_k^f(T,\cdot)\right\|_{L^2(0,1)}^2
= 
 \left\|\left(1_{(0,s)} - 1\right) u_k^{f}(T,\cdot)\right\|_{L^2(0,1)}^2.
	\end{align}

Let $f_j \in C_0^\infty(T-s, T)$ satisfy \eqref{conv_cutoff}.
Then
\begin{multline*}
	\lim_{j\rightarrow\infty}\left\|u_1^{f_j}(T,\cdot) - u_1^{f}(T,\cdot)\right\|_{L^2(0,1)}^2 
	=\inf_{\tilde{f}\in C_0^\infty(T - s,T)} \left\| u_1^{\tilde{f}}(T,\cdot) - u_1^f(T,\cdot)\right\|_{L^2(0,1)}^2.
\end{multline*}
Due to Lemma \ref{scalar lec2}, 
\begin{multline*}
	\lim_{j\rightarrow\infty}\left\|u_2^{f_j}(T,\cdot) - u_2^{f}(T,\cdot)\right\|_{L^2(0,1)}^2 
	=\inf_{\tilde{f}\in C_0^\infty(T - s,T)} \left\| u_2^{\tilde{f}}(T,\cdot) - u_2^f(T,\cdot)\right\|_{L^2(0,1)}^2.
\end{multline*}
Using \eqref{k1d} and \eqref{inf_proj} we see that
\begin{align*}
&\left\|u_2^{f_j}(T,\cdot) - 1_{(0,s)}u_2^{f}(T,\cdot)\right\|_{L^2(0,1)}^2 
\\&\quad
=
	\left\|u_2^{f_j}(T,\cdot) - u_2^{f}(T,\cdot)\right\|_{L^2(0,1)}^2  
	- \left\|\left(1_{(0,s)} - 1\right) u_2^{f}(T,\cdot)\right\|_{L^2(0,1)}^2 \to 0.
\end{align*}
\end{proof}

\begin{corollary}\label{cor_scalar_prod_supp}
	Assume that $\Lambda_1 = \Lambda_2$. Let $f, h \in C_0^\infty(\mathbb{R}^+)$ and let $0<s<T$. Then
	\begin{align}\label{scalar prod supp 0}
		\left(1_{(0,s)}u_1^f (T, \cdot), u_1^h (T, \cdot)\right)_{L^2(0,1)} =
		 \left(1_{(0,s)}u_2^f (T, \cdot), u_2^h (T, \cdot)\right)_{L^2(0,1)}.
	\end{align}
\end{corollary}
\begin{proof}
By Lemma \ref{app con}, there is a sequence $\{f_j\} \subset C_0^\infty(T - s, T)$ such that 
	\begin{equation*}
		u_k^{f_j} (T,\cdot) \rightarrow 1_{(0,s)}u_k^f (T, \cdot)
	\end{equation*}
in $L^2(0,1)$ for $k=1$. By Lemma \ref{lem_of_norm_conv_of_sol_1d}
this holds also for $k=2$. In view of Lemma~\ref{scalar lec2}, we have
\begin{multline*}
	\left(1_{(0,s)} u_1^{f}(T,\cdot), u_1^h(T,\cdot)\right)_{L^2(0,1)} = \lim_{j\rightarrow \infty}\left(u_1^{f_j}(T,\cdot), u_1^h(T,\cdot)\right)_{L^2(0,1)}\\
	= \lim_{j\rightarrow \infty}\left(u_2^{f_j}(T,\cdot), u_2^h(T,\cdot)\right)_{L^2(0,1)} = \left(1_{(0,s)} u_2^{f}(T,\cdot), u_2^h(T,\cdot)\right)_{L^2(0,1)}.
\end{multline*}
\end{proof}

We will need the following lemma. Its proof is given in Section \ref{sec_go}. 

\begin{lemma}\label{lem_nonvanishing}
	Let $T\geq 1$, then for any $x_0\in (0,1)$ there is $f\in C_0^\infty(0,T)$ such that $u^f(T,x_0) \neq 0$. 
\end{lemma}

Now we are ready to prove the main result of this section
\begin{theorem}\label{th_1d_main}
	If $\Lambda_1=\Lambda_2$, then $q_1=q_2$.
\end{theorem}
\begin{proof}
We consider $T > 1$ and let $x\in (0,1)$. Let us differentiate the left-hand side of \eqref{scalar prod supp 0}:
\begin{multline*}
	\partial_x 	\left(1_{(0,x)}(\cdot)u_1^f (T, \cdot), u_1^h (T, \cdot)\right)_{L^2(0,1)} \\
	= \partial_x\int_{0}^{x} u_1^f (T, y) u_1^h (T, y) dy = u_1^f (T, x) u_1^h (T, x).
\end{multline*}
Therefore, by \eqref{scalar prod supp 0}, we obtain
\begin{equation}\label{bububu}
	u_1^f (T, x) u_1^h (T, x) = u_2^f (T, x) u_2^h (T, x).
\end{equation}

Due to Lemma \ref{lem_nonvanishing},
for each $x \in (0,1)$ there is $f \in C_0^\infty(\R^+)$ such that 
$u_1^f(T,x) \ne 0$. Choosing such $f$ we may define
	\begin{align*}
w(x) = \frac{u_2^f (T, x)}{u_1^f (T, x)}.
	\end{align*}
We emphasize that the choice of $f$ depends on $x$, and it may appear that $w$ could be non-smooth. However, it is smooth. Indeed, 
	\begin{align}\label{w_eq}
u_1^h (T, x) = w(x) u_2^h (T, x),
	\end{align}
and for each $x_0 \in (0,1)$ there is $h \in C_0^\infty(\R^+)$ such that 
$u_2^h(T,x_0) \ne 0$. Thus, for $x$ near $x_0$,
	\begin{align*}
w(x) = \frac{u_1^h (T, x)}{u_2^h (T, x)}.
	\end{align*}
As the right-hand side is smooth for $x$ near $x_0$, and as $x_0 \in (0,1)$ is arbitrary, we see that $w$ is smooth. 

Let us now take $f = h$ in \eqref{bububu} and use \eqref{w_eq},
	\begin{align*}
(u_2^h(T,x))^2 = (u_1^h(T,x))^2 = w^2(x) (u_2^h(T,x))^2.
	\end{align*} 
Choosing again $x \in (0,1)$ and $h \in C_0^\infty(\R^+)$ such that 
$u_2^h(T,x) \ne 0$, we see that $w^2(x) = 1$.
The smoothness of $w$ implies that it is a constant function
taking the value $1$ or $-1$. 

To summarize 
	\begin{align*}
u_1^h (T, x) = \pm u_2^h (T, x),
	\end{align*}
for all $T > 1$, $x \in (0,1)$ and $h \in C_0^\infty(\R^+)$.

There holds
	\begin{align*}
0 = (\partial_t^2 - \partial_x^2 + q_1) u_1^h 
= \pm (\partial_t^2 - \partial_x^2 + q_1) u_2^h 
= \pm (q_1 - q_2) u_2^h.
	\end{align*}
Choosing such $h \in C^\infty_0(\R^+)$ that $u_2^h(T,x) \ne 0$, we get 
$q_1(x) = q_2(x)$.
\end{proof}

\section{Boundary Control method on a Riemannian manifold}

The main advantage of the Boundary Control method is that it works in general geometric settings. 
We refer to \cite{Lee} for an introduction to Riemannian geometry. A reader, who is not interesting in geometry, may simply replace $(M,g)$ by a domain $M \subset \R^n$ with a smooth boundary $\partial M$.

Let $(M,g)$ be a compact Riemannian manifold with boundary and $q\in C^{\infty}(M)$. Consider the wave equation 
    \begin{align}\label{main_eq_gen}
\begin{cases}
\partial_t^2 u - \Delta_g u + q u = 0 & \text{on } \mathbb{R}^+ \times M,
\\
\left.u\right|_{x \in \partial M} = f,
\\
\left.u\right|_{t = 0} = \left.\partial_t u\right|_{t=0} = 0.
\end{cases}
    \end{align}
Here $\Delta_g$ is the Laplacian. In local coordinates, it is given as follows 
\begin{equation*}
	\Delta_g u = \frac{1}{\sqrt{\textrm{det}(g)}} \sum_{i,j=1}^{n} \partial_i \left(\sqrt{\textrm{det}(g)} g^{ij}\partial_j u\right).
\end{equation*}
In particular, in the case that $M \subset \R^n$, the above expression gives the usual Laplacian, since $g$ is the identity matrix.

We define the Dirichlet-to-Neumann map 
	\begin{align*}
\Lambda: C_0^\infty(\mathbb{R}^+\times \partial M) \rightarrow C^\infty(\mathbb{R}^+\times \partial M)
	\end{align*}
 as follows
\begin{equation*}
	\Lambda f =\left.\partial_\nu u^f\right|_{\mathbb{R}^+\times\partial M},
\end{equation*}
where $u^f$ is the solution in of \eqref{main_eq_gen}.
Here $\nu$ is the exterior unit normal on $\partial M$.
We will study the inverse problem: determine the potential $q$ given $\Lambda$ and the Riemannian manifold $(M,g)$.

The finite speed of propagation will be formulated in terms of the  distance function $d_g$ on $(M,g)$. That is, it takes the time $d_g(x,y)$ for a wave to propagate from a point $x \in M$ to a point $y \in M$.
We recall the definition of the distance on Riemannian manifold. 

If $(M,g)$ is is a connected Riemannian manifold and $x,y\in M$, the
distance between $x$ and $y$ is defined by
\begin{equation*}
	d_g(x,y) = \inf \{ l(\gamma):  \gamma\in C_{x,y}\},
\end{equation*}
where 
\begin{align*}
	C_{x,y} = \{ \gamma:[0,\ell] \rightarrow M: \gamma &\text{ is a piecewise smooth, continuous curve and }  \\&\gamma(0) = x, \; \gamma(\ell) = y \}
\end{align*}
and
\begin{equation*}
	l(\gamma) = \int_{0}^{\ell}  \sqrt{g(\dot{\gamma}(t), \dot{\gamma}(t))} dt. 
\end{equation*}
Here $g$ is viewed as a bilinear form. 
When $M \subset \R^n$ is convex, $d_g$ is the usual Euclidean distance.

Before entering into the multidimensional geometric case, let us see how non-constant speed of sound gives the speed of propagation in $1+1$ dimensions.

\subsection{Finite speed of propagation with nonconstant speed of sound}

Let $c \in C^\infty([0,1])$ and suppose that $c(x) > 0$ for all $x \in [0,1]$.
Consider a solution $u$ to 
    \begin{align}\label{eq_wave_c}
		(\partial_t^2 - c^2 \partial_x^2) u = 0 \quad \text{on } \mathbb{R}^+\times (0,1),
    \end{align}
which satisfies 
\begin{equation*}
	\left.u\right|_{x=1} = 0.
\end{equation*}
We set
\begin{equation*}
	\mathcal{E}(t,x)  = c^{-2}(x)|\partial_t u(t,x)|^2 + |\partial_x u(t,x)|^2.
\end{equation*}
We want to find an increasing function $r \in C^\infty(\R)$ with $r(0) = 0$ such that the energy
\begin{equation*}
	E(t) = \frac{1}{2} \int_{r(t)}^{1} \mathcal{E}(t,x) dx
\end{equation*}
satisfies
\begin{equation}\label{partial E}
	\partial_tE(t) \leq 0.
\end{equation}
The slower $r$ increases, the better finite speed of propagation result we will get. 

We write, by using Leibniz integral rule,
\begin{align*}
	\partial_t E(t) &= - \frac{1}{2} r'(t) \mathcal{E}(t, r(t)) + \frac{1}{2} \int_{r(t)}^{1} \partial_t \mathcal{E} (t,x)dx \\
	& = - \frac{1}{2} r'(t) \mathcal{E}(t, r(t)) + \int_{r(t)}^{1}\left(  c^{-2} (x) \partial_t u(t,x) \partial_t^2 u(t,x) + \partial_x u(t,x) \partial_{tx} u(t,x)\right) dx.
\end{align*}
Integration by parts gives
\begin{multline*}
	\partial_t E(t) = - \frac{1}{2} r'(t) \mathcal{E}(t, r(t)) + \left[\partial_x u(t,x)\partial_t u(t,x)\right]^{1}_{r(t)}\\
	+ \int_{r(t)}^{1} \left(c^{-2} \partial_t^2u(t,x) + \partial_x^2u(t,x)\right) \partial_t u(t,x) dx.
\end{multline*}
Since $u$ is the solution of the wave equation the last integral is $0$. 
Moreover, due to the boundary condition, the second term is $0$ at $x=1$, so that
\begin{equation*}
	\partial_t E(t) = -\frac{1}{2} r'(t) \mathcal{E}(t,r(t)) - \left.\left(\partial_tu(t,x)\partial_x u(t,x)\right)\right|_{x = r(t)}.
\end{equation*}
Note that $r'>0$ since $r$ is increasing. Therefore, using a simple inequality 
\begin{equation*}
	2 xy \leq \alpha x^2 + \frac{1}{\alpha} y^2 , \qquad \text{for } \alpha, x, y>0,
\end{equation*}
we know that \eqref{partial E} holds if
\begin{equation*}
	\frac{1}{r'(t)} = r'(t) c^{-2}(r(t)),
\end{equation*}
or equivalently, 
\begin{equation}\label{ode r}
	r'(t) = c(r(t)).
\end{equation}
This equation is solvable. Indeed, consider the following function
\begin{equation*}
	\rho(x)  = \int_0^x \frac{1}{c(y)} dy.
\end{equation*}
Since $\rho$ is a strictly increasing function, its inverse function exists, so we can set
\begin{equation*}
	r(t) = \rho^{-1} (t).
\end{equation*}
It is easy to check that this function indeed satisfies \eqref{ode r}. Observe that if $c=1$ identically, we obtain $r(t) = t$.

The function $\rho$ is the travel time between $0$ and $x$: the time necessary for perturbation at $0$ to reach $x$. It also can be interpreted as the distance from $0$ to $x$.
Indeed, let us consider the metric
\begin{equation*}
	g = c^{-2} dx,
\end{equation*}
on $[0,1]$. Then, for $x\in [0,1]$, we get
\begin{equation*}
	d_g(0,x) = \inf_{\gamma\in C_{0,x}} \int_{0}^{b}  \sqrt{g_{\gamma(t)}(\dot{\gamma}(t), \dot{\gamma}(t))} dt.
\end{equation*}
Let $\gamma$ be a minimizer curve. Then $\gamma :[0,b] \rightarrow [0,x]$ needs to be a bijection. By changing the coordinates $\tau = \gamma(t)$, we obtain
\begin{equation*}
	d_g(0,x) = \int_{0}^{b} \dot{\gamma}(t) \frac{1}{|c(\gamma(t))|}dt = \int_{0}^{x} \frac{1}{|c(\tau)|} d\tau = \rho(x).
\end{equation*}

\subsection{Main tools}

As in the one dimensional case, the main ingredients of solving the inverse problem we are considering here are the finite speed of propagation and unique continuation. 

\begin{theorem}[Finite speed of propagation]
	Let $\Omega\subset M$ be open, and define 
	\begin{equation*}
		\mathcal{C}= \{(t,x)\in \mathbb{R}\times M: d_g(x, M\setminus\Omega)>|t|\}.
	\end{equation*}
Suppose that $u \in C^2(\R \times M)$ is a solution of 
	\begin{equation*}
		\begin{cases}
			\left(\partial_t^2 - \Delta_g + q(x)\right) u = 0, & \text{on } \mathbb{R}\times M;\\
			\left.u\right|_{t=0} = \left.\partial_t u\right|_{t=0} = 0, & \text{on } \Omega.
		\end{cases}
	\end{equation*}
Suppose, furthermore, that $\left.u\right|_{\mathcal C \cap (\R \times \partial M)} = 0$.
	Then $\left.u\right|_{\mathcal{C}} = 0$.
\end{theorem}

We omit the proof and refer to \cite[Theorem 2.47]{Matti}.

\begin{theorem}[Unique continuation]\label{uniq_con_nd}
Let $\Gamma\subset \partial M$ be open, and define
	\begin{equation*}
		K = \{(t,x) \in \R\times M: d_g(x,\Gamma) \leq T - |t|\}.
	\end{equation*}
Suppose that $u\in C^2(\R\times M)$ is a solution of 
	\begin{equation*}
		\begin{cases}
			\left(\partial_t^2 - \Delta_g + q(x)\right) u = 0, & \text{on } \mathbb{R}\times M;\\
		\left.u\right|_{[-T,T] \times \Gamma} = \left.\partial_\nu u\right|_{[-T,T] \times \Gamma} = 0.
		\end{cases}
	\end{equation*}
Then $\left.u\right|_{K} = 0$.
\end{theorem}

We omit the proof and refer to \cite[Theorem 3.16]{Matti}. An introduction to the concept of unique continuation for the wave operator and its applications can be found in \cite{LaurentLeautaud}. Applying Theorem \ref{uniq_con_nd} on $M \setminus \Omega$, where $\Omega \subset M$ is a small open set with smooth boundary satisfying $\partial \Omega \cap \partial M = \emptyset$, allows us to compare unique continuation to finite speed of propagation, see Figure \ref{nd pic}. Clearly these two results are genuinely different in the higher dimensional case, in contrast to the one dimensional case. (Recall that in the dimension one, finite speed of propagation and unique continuation differ only by interchanging space and time.)

\begin{figure}%
	\centering
	\subfloat[\centering Finite speed of the wave propagation]{{\includegraphics[width=6.2cm]{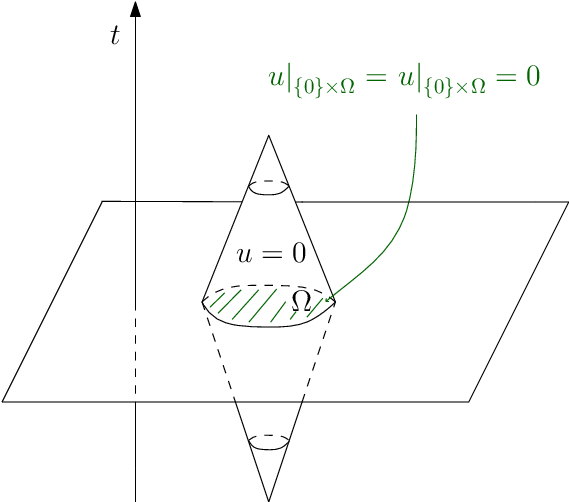} }}%
	\qquad
	\subfloat[\centering Unique continuation]{{\includegraphics[width=6.2cm]{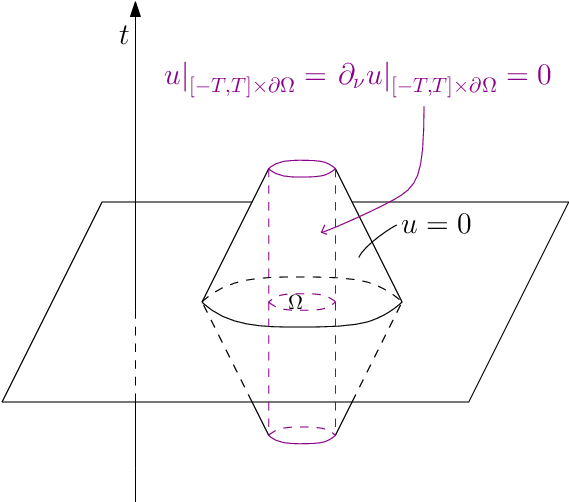} }}%
	\caption{$n + 1$ dimensional case}%
	\label{nd pic}%
\end{figure}

Let $\Gamma\subset \partial M$ be an open set. We define the domain of influence
\begin{equation*}
	M(\Gamma,T) = \{x\in M: d_g(x,\Gamma)\leq T\}
\end{equation*}
and use the identification
\begin{equation*}
	L^2(M(\Gamma,T)) = \{\phi\in L^2(M): \supp(\phi)\subset M(\Gamma,T)\}.
\end{equation*}

Similarly, as we derived Lemma \ref{app con} from Theorem \ref{uniq_con_1d}, Theorem \ref{uniq_con_nd} gives us the following approximate controlability:

\begin{lemma}\label{app con nd}
	Let $T>0$, then the set 
	\begin{equation*}
		\mathcal{B}(T, \Gamma) = \{u^f(T,\cdot): f\in C_0^\infty((0,T)\times \Gamma)\}
	\end{equation*}
	is a dense subset of $L^2(M(\Gamma,T))$.
\end{lemma}

Lemma \ref{lem_nonvanishing} can be generalized for the multidimensional case:

\begin{lemma}\label{lem_nonvanishing_gen}
	Let $T>0$, then for any point $x_0$ in the interior of $M(\Gamma, T)$ there is $f\in C_0^\infty ((0, T)\times \Gamma)$ such that $u^f(T,x_0)\neq 0$.
\end{lemma}

\subsection{Solution to the inverse problem}

Let $T>0$. For functions $f$, $h\in C_0^\infty(\R^+ \times \partial M)$, we define
\begin{equation*}
	W_{f,h}(t,s) = (u^f(t,\cdot), u^h(s,\cdot))_{L^2(M)}.
\end{equation*}
The following lemma is the higher dimensional analogue of Lemma \ref{scalar lec2}.
\begin{lemma}\label{scalr_prod_nd}
	Let $f$, $h\in C_0^\infty(\R^+ \times \partial M)$. The operator $\Lambda$ determines $W_{f,h}(t,t)$ for all $t > 0$.
\end{lemma}
\begin{proof}
	We write
	\begin{multline*}
		(\partial_t^2 - \partial_s^2) W_{f,h}(t,s) = \left( \Delta_g u^{f}(t,\cdot) - q u^f(t,\cdot), u^{h}(s,\cdot) \right)_{L^2(M)}\\
		- \left( u^{f}(t,\cdot), \Delta_g u^{h}(s,\cdot) - q u^h(s,\cdot) \right)_{L^2(M)}.
	\end{multline*}
	Further, by Green's identity, we obtain
	\begin{multline*}
		(\partial_t^2 - \partial_s^2) W_{f,h}(t,s)  = \int_{\partial M} \partial_\nu u^f(t,x) u^h(s,x) dS(x) - \int_{\partial M}  u^f(t,x) \partial_\nu u^h(s,x) dS(x)\\
		= \left( \Lambda f(t,\cdot), h(s,\cdot)\right)_{L^2(\partial M)} - \left( f(t,\cdot),  \Lambda h(s,\cdot)\right)_{L^2(\partial M)}.
	\end{multline*}
	Let us denote the right-hand side by $F(t,s)$, then $W_{f,h}$ is the solution of the equation
	\begin{equation*}
		\begin{cases}
			(\partial_t^2 - \partial_s^2) W_{f,h}(t,s) = F(t,s) & \text{on } (0,T)\times (0, T),\\
			W_{f,h}(0,s) = \partial_tW_{f,h}(0,s) = 0.
		\end{cases}
	\end{equation*}
	Hence, $W_{f,h}$ is determined by $F$, and consequently, it is determined by $\Lambda$. 
\end{proof}

We have the analogue of Lemma \ref{lem_of_norm_conv_of_sol_1d}.

\begin{lemma}\label{eq_of_norm_conv_of_sol}
Assume that $\Lambda_1 = \Lambda_2$.
	Let $0<s<T$, let $\Gamma \subset \partial M$ be open and let $\{f_j\} \subset C_0^\infty((T-s, T)\times \Gamma)$ be a sequence such that
	\begin{equation*}
		u_1^{f_j}(T,\cdot) \rightarrow 1_{M(\Gamma,s)}u_1^f(T,\cdot) \qquad \text{in } L^2(M).
	\end{equation*}
	Then
	\begin{equation*}
		u_2^{f_j}(T,\cdot) \rightarrow 1_{M(\Gamma,s)}u_2^f(T,\cdot) \qquad \text{in } L^2(M).
	\end{equation*}
\end{lemma}

We omit the proof as it coincides with that in the one dimensional case. As before, we have also the following corollary 

\begin{corollary}\label{cor_scal_prod_sol_dominf_nd}
Assume that $\Lambda_1 = \Lambda_2$.
	Let $0< s< T$, $t > 0$ and let $\Gamma\subset \partial M$ be open, then for any $f$, $h\in C_0^\infty(\R^+\times \partial M)$, it follows
	\begin{equation*}
		\left(1_{M(\Gamma,s)}u_1^f(T,\cdot), u_1^h(t,\cdot)\right)_{L^2(M)} = \left(1_{M(\Gamma,s)}u_1^f(T,\cdot), u_1^h(t,\cdot)\right)_{L^2(M)}.
	\end{equation*}
\end{corollary}

Now we deviate from the one dimensional proof.

\begin{corollary}\label{cor_scal_prod_sol_dominf_second}
Assume that $\Lambda_1 = \Lambda_2$.
	Let $s, \tilde s \in (0,T)$, let $t > 0$ and consider two open $\Gamma$, $\tilde{\Gamma}\subset \partial M$. Then we have for any $f$, $h\in C_0^\infty(\R^+\times \partial M)$
	\begin{multline*}
		\left(1_{M(\Gamma,s)} 1_{M(\tilde{\Gamma},\tilde{s})} u_1^f(T,\cdot), u_1^h(t,\cdot)\right)_{L^2(M)} \\
		= \left(1_{M(\Gamma,s)}1_{M(\tilde{\Gamma},\tilde{s})}u_1^f(T,\cdot), u_1^h(t,\cdot)\right)_{L^2(M)}.
	\end{multline*}
\end{corollary}

\begin{proof}
	By Lemmas \ref{app con nd} and \ref{eq_of_norm_conv_of_sol}, there is a sequence $\{f_j\} \subset C_0^\infty ((T-s, T)\times \Gamma)$ such that
	\begin{equation*}
		u_1^{f_j}(T,\cdot) \rightarrow 1_{M(\Gamma,s)} u_1^{f}(T,\cdot) \quad \text{and} \quad
		u_2^{f_j}(T,\cdot) \rightarrow 1_{M(\Gamma,s)} u_2^{f}(T,\cdot).
	\end{equation*}
	Therefore, using Corollary \ref{cor_scal_prod_sol_dominf_nd}, we obtain
	\begin{align*}
		\left(1_{M(\Gamma,s)} 1_{M(\tilde{\Gamma},\tilde{s})} u_1^f(T,\cdot),u_1^h(t,\cdot)\right)_{L^2(M)} &\quad =  \lim_{j\rightarrow \infty} \left(1_{M(\tilde{\Gamma},\tilde{s})}u_1^{f_j}(T,\cdot), u_1^h(t,\cdot)\right)_{L^2(M)}\\
		& \quad =\lim_{j\rightarrow \infty} \left(1_{M(\tilde{\Gamma},\tilde{s})} u_2^{f_j}(T,\cdot), u_2^h(t,\cdot)\right)_{L^2(M)}\\
		& \quad =\left(1_{M(\Gamma,s)} 1_{M(\tilde{\Gamma},\tilde{s})} u_1^f(T,\cdot), u_1^h(t,\cdot)\right)_{L^2(M)} .
	\end{align*}
\end{proof}

\begin{theorem}
	If $\Lambda_1 = \Lambda_2$, then $q_1 = q_2$.
\end{theorem}

\begin{proof}
	Let $x_0$ be an interior point of $M$ and write $s=d_g(x_0,\partial M)$. We choose $T>s$. Let $y\in \partial M$ be such that
	\begin{equation*}
		d_g(x,y) = s.
	\end{equation*}
	Let $\Gamma \subset \p M$ be a neighbourhood of $y$ and take  $\tilde \Gamma = \partial M$.
 Let also $0 < \tilde{s} < s $ and set
	\begin{equation*}
		Z = M(\Gamma,s) \setminus M(\tilde{\Gamma},\tilde{s}).
	\end{equation*}
	Then, 
	\begin{multline*}
		\frac{1}{|Z|} \left(\left(  1_{M(\Gamma,s)}  u_1^f(T,\cdot), u_1^h(t,\cdot)  \right)_{L^2(M)}  - \left(  1_{M(\tilde{\Gamma},\tilde{s})} 1_{M(\Gamma,s)} u_1^f(T,\cdot),  u_1^h(T,\cdot)  \right)_{L^2(M)} \right)\\
		\rightarrow u_1^f(T,x_0)u_1^h(T,x_0)
	\end{multline*}
	as $\tilde{s} \rightarrow s$ and $\Gamma \to \{y\}$. The same holds for $u_2^f$ and $u_2^h$, so that by Corollaries \ref{cor_scal_prod_sol_dominf_nd} and \ref{cor_scal_prod_sol_dominf_second}, we know that 
	\begin{equation*}
		u_1^f(T,x_0)u_1^h(t,x_0) = u_2^f(T,x_0)u_2^h(t,x_0).
	\end{equation*}
This is the analogue of \eqref{bububu}, and we conclude as in the proof of Theorem \ref{th_1d_main}.
\end{proof}

\section{Geometric optics}\label{sec_go}

In this section, we will construct solutions to 
\begin{align}\label{eq_wave_Min}
&\p_t^2 u - \Delta u + q u = 0, \quad \text{in $(0,T) \times \R^n$},
\end{align}
that concentrate on light rays, that is, lines of the form 
$$
\beta(s) = (s, y + s v), \quad s \in \R,
$$
where $y$ is a point in $\R^n$ and $v$ is a unit vector in $\R^n$.
We write $$S^{n-1} = \{v \in \R^n;\ |v|=1\}$$
for the set of unit vectors, that is, for the unit sphere. 
The name light ray comes from the fact that the tangent vector $\dot \beta = (1, v)$ is light like with respect to the Minkowski metric
\begin{equation}\label{minkowski_g}
g = \begin{pmatrix}
-1 \\
& 1 \\
&& \ddots \\
&&& 1 
\end{pmatrix},
\end{equation}
that is, $g(\dot \beta, \dot \beta) = 0$.

The idea is to find first an approximate solution of the form 
$$
e^{i \sigma \phi(t,x)} (a_0(t,x) + \sigma^{-1} a_1(t,x) + \sigma^{-2} a_2(t,x) + \dots), 
$$
where $\sigma > 0$ is a large parameter, 
and then an actual solution 
$$
u = e^{i \sigma \phi} (a_0 + \dots) + r_\sigma,
$$ 
where the remainder $r_\sigma$ converges to zero as $\sigma \to \infty$.
We will begin with the single term approximation $e^{i \sigma \phi} a_0$ and write $a_0 = a$.

\subsection{Single term ansatz}

To simplify the notation, we write $\Box = \p_t^2 - \Delta$.
The equation $(\Box + q) u = 0$ is equivalent to
\begin{equation}\label{wave_eq_r}
(\Box + q) r_\sigma = - (\Box + q) (e^{i\sigma \phi} a),
\end{equation}
and we want to choose $\phi$ and $a$ so that 
\begin{align}\tag{``C''}
\Box (e^{i\sigma \phi} a) = e^{i\sigma \phi} \Box a.
\end{align}
The rationale is that in this case the absolute value of the right-hand side of (\ref{wave_eq_r}) is independent of $\sigma$, and therefore $r_\sigma$ is at least not blowing up as $\sigma \to \infty$.

It is a simple matter to expand the left-hand side of (``C'') but a useful computational technique is to consider the conjugated wave operator 
$$
e^{-i\sigma \phi} \Box e^{i\sigma \phi}
= e^{-i\sigma \phi} \p_t^2 e^{i\sigma \phi} + \dots
= e^{-i\sigma \phi} \p_t  e^{i\sigma \phi} e^{-i\sigma \phi} \p_t  e^{i\sigma \phi} + \dots.
$$
Now $e^{-i\sigma \phi} \p_t  e^{i\sigma \phi} = \p_t + i \sigma (\p_t\phi)$ and 
$$
(e^{-i\sigma \phi} \p_t e^{i\sigma \phi})^2
= \p_t^2 + 2 i \sigma (\p_t\phi) \p_t - \sigma^2 |\p_t \phi|^2 + i \sigma (\p_t^2 \phi).
$$
Treating the spacial derivatives in the same way we get 
\begin{align}\label{conjugation}
e^{-i\sigma \phi} \Box e^{i\sigma \phi} = 
\Box + i \sigma ( 2 (\p_t\phi) \p_t - 2 (\nabla \phi) \cdot \nabla + \Box \phi) - \sigma^2 (|\p_t \phi|^2 - |\nabla \phi|^2).
\end{align}
Therefore for $a \ne 0$, (``C'') is equivalent to the following two equations
\begin{align}
\tag{E}|\p_t \phi|^2 - |\nabla \phi|^2 = 0,
\\\tag{T}
2 (\p_t\phi) \p_t a - 2 (\nabla \phi) \cdot \nabla a + (\Box \phi) a = 0.
\end{align}

It is natural to normalize $\phi$ so that (E) becomes 
$|\p_t \phi|^2 = |\nabla \phi|^2 = 1$.
There is some freedom when choosing a solution to (E), but for our purposes it suffices to use the linear solution $\phi(t,x) = t + v \cdot x$ where $v \in S^{n-1}$.

To simplify the notation, we may assume after a rotation that $v \cdot x = -x^1$. The functions satisfying $|\nabla \phi| = 1$ are often called distance functions, and the particular choice $x^1$ is of course the signed distance to the plane $x^1 = 0$. 

The transport equation (T) simplifies now to 
$$
\p_t a + \p_{x^1} a = 0.
$$
The solutions to this are of the form $$a(t,x) = \chi(t - x^1) \eta(x')$$ where $x' = (x^2, x^3, \dots, x^n)$.
Taking $\chi \approx \delta$ and $\eta \approx \delta$ we obtain $a$ that concentrates on the light ray $\beta(s) = (s,s,0)$.

Neither $\phi$ nor $a$ depend on $q$ in the above construction. In order to obtain information on $q$, in the context of inverse problems, there are two typical approaches: use the difference of two solutions, corresponding to different potentials, or use a multi-term approximation. We will take the latter approach.

\subsection{Multi-term ansatz}

Let us consider the ansatz,
$$e^{i \sigma \phi} A, \quad A = a_0 + \sigma^{-1} a_1 + \dots + \sigma^{-N} a_N,$$
and choose $\phi$ and $a_0 = a$ as above.
As we are using a more complicated amplitude, we can ask for more than (``C''), namely 
\begin{align*}
(\Box + q) (e^{i\sigma \phi} A) = \mathcal O(\sigma^{-N}), 
\quad \sigma \gg 1. 
\end{align*}
We use the conjugation formula (\ref{conjugation}), to obtain
$$
e^{-i\sigma \phi} (\Box + q) e^{i\sigma \phi} A
= (\Box + q) A + 2i (\p_t + \p_{x^1}) (a_1 + \dots + \sigma^{-N+1} a_N).
$$
This is of order $\sigma^{-N}$ whenever $a_1,\dots,a_N$ solve the transport equations
$$
\p_t a_j + \p_{x^1} a_j - \frac i 2 (\Box + q) a_{j-1} = 0, \quad j=1,\dots,N,
$$
or after the change of variables,
$$
s = \frac{t+x^1} 2, \quad r = \frac{t-x^1} 2,
$$
equivalently $\p_s a_j = \frac i 2 (\Box + q) a_{j-1}$.
Therefore we may choose
	\begin{align}\label{def_aj}
a_j(s,r,x') = \frac i 2 \int_{-r}^s (\Box + q) a_{j-1}(s',r,x') ds'.
	\end{align}
Note that $t=0$ is equivalent to $s = -r$. The choice of the lower limit $-r$ in the integration implies that $a_j = 0$, $j=1,\dots,N$, when $t=0$.

\subsection{Solving for the remainder}

When $\chi \in C_0^\infty(\R)$ and $\eta \in C_0^\infty(\R^{n-1})$, the restrictions of all the amplitudes $a_j$, $j=0,1,\dots,N$, are compactly supported in $(0,T) \times \R^n$.
We recall that the wave equation 
\begin{align*}
\begin{cases}
\Box u + q u = F & \text{in $(0,T) \times \R^n$},
\\
u|_{t=0} = \p_t u|_{t=0} = 0
\end{cases}
\end{align*}
has a unique solution $u$ satisfying 
$$
\norm{u}_{C(0,T; H^1(\R^n))} + 
\norm{u}_{C^1(0,T; L^2(\R^n))} \le C \norm{F}_{L^2((0,T) \times \R^n)},
$$
see e.g. \cite[Theorem 7.6]{Evans1998}.
We solve 
\begin{align*}
\begin{cases}
\Box r_\sigma + q r_\sigma = -(\Box + q) (e^{i\sigma \phi} A) & \text{in $(0,T) \times \R^n$},
\\
r_\sigma|_{t=0} = \p_t r_\sigma|_{t=0} = 0.
\end{cases}
\end{align*}
As the right-hand side is pointwise of order $\sigma^{-N}$ and compactly supported, we see that $r_\sigma|_{t=T} = \mathcal O(\sigma^{-N})$ in $H^1(\R^n)$.

\subsection{Proof of Lemma \ref{lem_nonvanishing}}

If $\chi(0) \ne 0$ and $\eta(0) \ne 0$ then $a$ does not vanish anywhere on the line given by $t=x^1$.
For large enough $\sigma > 0$, the same is true for $e^{i \sigma \phi} A$ when restricted on $[0,T] \times \R^n$.

Let us now consider the case $n=1$.
The Sobolev embedding $H^1(\R) \subset C(\R)$, see e.g. \cite[Theorem 8.8]{Brezis}, implies that $r_\sigma|_{t=T} = \mathcal O(\sigma^{-N})$ in $C(\R)$.
Therefore $u(T,T) \ne 0$ for large $\sigma > 0$.
Lemma \ref{lem_nonvanishing} follows after a suitable translation in time. 

Lemma \ref{lem_nonvanishing_gen} can be shown in a similar manner, but higher regularity energy estimates, see e.g. \cite[Theorem 6, p. 412]{Evans1998}, are needed to bound $r_\sigma|_{t=T}$ in $H^k(\R^n)$ with $k$ large enough so that $H^k(\R^n) \subset C(\R^n)$.

\section{Solving an inverse problem using geometric optics}

We will consider the inverse problem to determine a compactly supported potential in a slab in the Minkowski space.
This avoids some technicalities appearing in more usual inverse boundary value problems, but it still allows us to introduce the main techniques. 

Let $n \ge 2$, $T > 0$, $q \in C_0^\infty((0,T) \times \R^n)$, and consider the wave equation 
\begin{align*}
\begin{cases}
\Box u + q u = 0 & \text{in $(0,T) \times \R^n$},
\\
u|_{t=0} = u_0,\ \p_t u|_{t=0} = u_1.
\end{cases}
\end{align*}
Define also the map 
$$
L : (C_0^\infty(\R^n))^2 \to C_0^\infty(\R^n),
\quad L (u_0, u_1) = u|_{t=T}.
$$
We will study how to solve the inverse problem to determine $q$ given $L$.

\subsection{Reduction to the light ray transform}
The function $u = e^{i \sigma \phi} A + r_\sigma$ constructed in Section \ref{sec_go}
solves (\ref{eq_wave_Min}) with 
$$
u_0 = (e^{i \sigma \phi} A)|_{t=0}, \quad 
u_1 = \p_t (e^{i \sigma \phi} A)|_{t=0}.
$$
As $q$ vanishes near $t=0$, we see that $u_0$ and $u_1$ are independent from $q$.
This again implies that $L$ determines $u|_{t=T}$. 
Moreover, 
$$
\sigma (e^{-i \sigma \phi} u - a_0)|_{t=T}
\to a_1|_{t=T} \quad \text{as $\sigma \to \infty$}.
$$
This determines $a_1|_{t=T}$. 
Recalling \eqref{def_aj} and using the fact that $\Box a_0$ does not depend on $q$, we find the integral 
$$
\int_{-r}^{T-r} q a_{0}(s',r,x') ds'.
$$
Here we used the fact that $t=T$ is equivalent to $s = T-r$. 

In $(s,r,x')$ coordinates $a_0 = \chi(2 r) \eta(x')$ and the above integral reduces at $(r,x')=0$ to
$$
\int_{0}^{T} q(s', 0, 0) ds' \chi(0) \eta(0).
$$
As $q$ vanishes for $t < 0$ and $t > T$, we can recover
in $(s,r,x')$ coordinates $\int_\R q(s,0,0) ds$, or equivalently, in $(t,x^1,x')$ coordinates 
$$
\int_\R q (\beta (s)) ds, \quad \beta(s) = (s,s,0).
$$
Repeating the above argument, after using rotations and translations, we obtain the light ray transform of $q$,
$$
\mathcal L q(y,v) = \int_\R q (\beta_{y,v} (s)) ds, \quad \beta_{y,v}(s) = (s,y + sv),\ y \in \R^n,\ v \in S^{n-1}.
$$

\subsection{Inversion of the light ray transform}

The above reduction works also when $n=1$, but inversion of $\mathcal L$ requires $n \ge 2$. Indeed, if $n=1$, then $\mathcal L q = 0$ for $q(s,r) = q_0(s) q_1(r)$ with $q_j$, $j=0,1$, integrating to zero. 

For a fixed $v \in S^{n-1}$, consider the change of coordinates in $\R^{1+n}$,
$$(t,x) = (s,y+sv).$$
Then $y = x - t v$.
Using this, we obtain the Fourier slicing 
$$
\int_{\R^n} e^{-i \eta \cdot y} \mathcal Lq(y,v)\, dy
= \int_{\R^{1+n}} e^{-i\eta \cdot (x-t v)} q(t,x)\, dtdx
= \widehat q(-\eta \cdot v, \eta).
$$
Here $|\eta \cdot v| \le |\eta|$.
Moreover, as $n \ge 2$, we can choose a unit vector $w$ that is orthogonal to $\eta$. 
Then for $a \in [-1,1]$ and $\eta \ne 0$, we may choose 
$$
v = - \frac a {|\eta|} \eta + \sqrt{1-a^2}\, w \in S^{n-1}.
$$ 
This gives $-\eta \cdot v = a |\eta|$,
and we see that the Fourier slicing allows us to recover 
$\widehat q(a |\eta|, \eta)$ for any $a \in [-1,1]$.

As $q$ is compactly supported, $\hat q$ is analytic. We know $\hat q$ in a non-empty open cone (in fact, in the cone of spacelike directions), and therefore everywhere by analytic continuation. This shows that $\mathcal L q$ determines $q$. By the above reduction also $L$ determines $q$.

\HOX{Unify the style for refs (maybe using bibtex), and sort them alphabetically}

\bibliographystyle{plain}
\bibliography{references}

\setlength{\parskip}{0pt}

\end{document}